\documentclass[11pt]{article}
\usepackage{amscd,amsmath,amstext,amsfonts,amsbsy,amssymb,amsthm,mathrsfs}
\usepackage{multicol,graphicx,float}
\usepackage{array,multirow}
\usepackage{fancyhdr,booktabs}
\usepackage[unicode,colorlinks,citecolor=red,linkcolor=blue]{hyperref}

\newtheorem{theorem}{{\bf Theorem}}[section]
\theoremstyle{definition} 

\theoremstyle{plain} \newtheorem{lemma}[theorem]{Lemma}

\title{\bf A Two Dimensional Backward Heat Problem With Statistical Discrete Data}
\author{ Nguyen Dang Minh$^1$\footnote{Corresponding author: \url{ndminhsv1986@gmail.com}}, To Duc Khanh$^2$, Nguyen Huy Tuan$^1$ and Dang Duc Trong$^1$ \\\\
\small $^1$ Fact. Maths and Computer Science, University of Science,\\
\small Vietnam National University, 227 Nguyen Van Cu, Dist.5, HoChiMinh City, VietNam.\\
\small $^2$ Dept. Statistical Sciences, University of Padua, 
\\
\small via C. Battisti, 241; 35121 Padua, Italy.
}
\begin{document}
\date{}
\maketitle
\begin{abstract}
In this paper, we focus on the backward heat problem of finding the function $\theta(x,y)=u(x,y,0)$ such that
\begin{equation*}
\left\{\begin{array}{l l l}
u_t - a(t)(u_{xx} + u_{yy}) & = f(x,y,t), & \qquad (x,y,t) \in \Omega\times (0,T),\\
u(x,y,T) & = h(x,y), & \qquad (x,y) \in\overline{\Omega},
\end{array}
\right.
\end{equation*}
where $\Omega = (0,\pi) \times (0,\pi)$ and the heat transfer coefficient $a(t)$ is known. In our problem, the source $f = f(x,y,t)$ and the final data $h(x,y)$ are unknown. We only know random noise data $g_{ij}(t)$ and $d_{ij}$ satisfying the regression models
\begin{eqnarray}
g_{ij}(t) &=& f(x_i,y_j,t) + \vartheta\xi_{ij}(t),\nonumber\\
d_{ij} &=& h(x_i,y_j) + \sigma_{ij}\epsilon_{ij},\nonumber
\end{eqnarray}
where $\xi_{ij}(t)$ are Brownian motions, $\epsilon_{ij}\sim \mathcal{N}(0,1)$, $(x_i,y_j)$ are grid points of $\Omega$ and $\sigma_{ij}, \vartheta$ are unknown positive constants. The noises $\xi_{ij}(t), \epsilon_{ij}$ are mutually independent. From the known data $g_{ij}(t)$ and $d_{ij}$, we can recovery the initial temperature $\theta(x,y)$. However, the result thus obtained is not stable and the problem is severely ill--posed. To regularize the instable solution, we use the trigonometric method in nonparametric regression associated with the truncated expansion method. In addition, convergence rate is also investigated numerically.
 

\end{abstract}
\section{Introduction}
In the literature of PDE research papers, the heat problem is pretty much studied, since it is important in various physics and industrial applications. The heat problem has many forms, in which we have two ordinary forms. The first one is of determining the future temperature of the system from the initial data. The second one is of finding the initial temperature from the final temperature. They are commonly known as ``forward heat problems'' and ``backward heat problems'', respectively. The backward heat problems are applied in fields as the heat conduction theory \cite{BBC}, material science \cite{RHN}, hydrology \cite{JB,LEP}, groundwater contamination \cite{SK}, digital remove blurred noiseless image \cite{CSH} and also in many other practical applications of mathematical physics and engineering sciences. As known, the backward heat problem is severely ill--posed (see \cite{r5} or Section 3).\\

Now, to consider the problem we state the precise form of our problem. Let $\Omega=(0,\pi)\times(0,\pi), T>0$ and $a:(0,T)\rightarrow\mathbb{R}$ be a positive Lebesgue measurable function. In this paper we focus on the two dimensional nonhomogeneous backward heat problems of finding functions $\theta(x,y):=u(x,y,0), f(x,y,t)$ and $h(x,y)$ such that
\begin{equation}
\label{prob2s}
\left\{\begin{array}{l l l}
u_t - a(t)(u_{xx} + u_{yy}) & = f(x,y,t), & \qquad (x,y,t) \in \Omega\times (0,T),\\
u(x,y,T) & = h(x,y), & \qquad (x,y) \in\overline{\Omega},
\end{array}
\right.
\end{equation}
subject to the Dirichlet condition
\begin{equation}
u(0,y,t) = u(\pi,y,t) = u(x,0,t) = u(x,\pi,t) = 0.
\label{cd1}
\end{equation}
We shall assume that $0 < a_1 \le a(t) \le a_2$ where $a_1, a_2$ are positive constants.\\

As known, measurements always are given at a discrete set of points
and contain errors. These errors may be generated from controllable sources or uncontrollable sources. In the first case, the error is often deterministic.  Hence, if we know approximation $f_\epsilon, h_\epsilon$ of the source $f$ and the final data $h$ then we can construct an approximation for $\theta(x,y)$. If the errors are generated from uncontrollable sources as wind, rain, humidity, etc, then the model is random. Methods for the deterministic cases cannot apply directly for the case.  Because of the random noise, the calculation is often intractable. In fact, let $(x_i,y_j) = \left(\frac{\pi (2i-1)}{2n},\frac{\pi (2j-1)}{2m}\right)$ with $i=\overline{1,n}; j=\overline{1,m}$, be grid points in $\Omega$. We consider two nonparametric regression models of data
\begin{eqnarray}
g_{ij}(t) &=& f(x_i,y_j,t) + \vartheta\xi_{ij}(t),\label{st_md}\\
d_{ij} &=& h(x_i,y_j) + \sigma_{ij}\epsilon_{ij}, \qquad i=\overline{1,n}; j=\overline{1,m} \label{md2}
\end{eqnarray}
where $g_{ij}(t)$ are random process, $d_{ij}$ are random data, $\xi_{ij}(t)$ are Brownian motions, $\epsilon_{ij}\sim \mathcal{N}(0,1)$ and $\sigma_{ij}$ are bounded by a positive constant $V_{max}$, i.e., $0 \leq \sigma_{ij} < V_{\text{max}}$ for all $i,j$. The random variables $\xi_{ij}(t), \epsilon_{ij}$ are mutually independent. In the model, $g_{ij}(t)$ and $ d_{ij}$ are observable whereas  
$\vartheta\xi_{ij}(t)$ and $\sigma_{ij}\epsilon_{ij}$ are unknown.
Now, we have a problem of finding the final temperature $h(x,y)$, the initial temperature $\theta(x,y)=u(x,y,0)$ and the source $f(x,y,t)$ from the random noise data $g_{ij}(t)$ and $d_{ij}$. \\

Roughly speaking, we can find two ways to examine the general backward problem: the numerical tendency and the theoretical tendency. In the numerical tendency, the paper is focused essentially on the new numerical algorithms and gives a lot of examples to convince readers about the effectiveness of the algorithms. In the theoretical tendency, authors have to consider the convergence of  algorithms  and give some illustration examples. Both approaches are important in real applications. On the other hand, we can classify informally problems into the deterministic and the stochastic problems. So we have four styles of investigating the problem
\begin{enumerate}
  \item {\bf N}umerical {\bf T}endency for {\bf D}eterministic {\bf P}roblem (NTDP for short),
  \item {\bf N}umerical {\bf T}endency for {\bf S}tochastic {\bf P}roblem (NTSP),
  \item {\bf T}heoretical {\bf T}endency for {\bf D}eterministic {\bf P}roblem (TTDP),
  \item {\bf T}heoretical {\bf T}endency for {\bf S}tochastic {\bf P}roblem (NTSP).
\end{enumerate}

The literature of the NTDP, TTDP styles for the backward problem is traditional and  very huge. For example, we have many approaches as Tikhonov method (\cite{r14,r13}), Quasi--boundary value method (see \cite{r16,r18,r19}), Quasi--Reversibility method (\cite{r20,show}), mollification (\cite{r21}), truncated expansion (\cite{PTN,r15}), the general filter regularization method (\cite{QW}).... The literature of the NTSP styles is also large. A lot of papers are of NTDP style but its examples have data with concrete random noises (often bounded with uniform distribution).  In the paper \cite{r141},  the authors used the backward group preserving scheme
to deal with the problem. The noisy data $R(i)$ in the paper are random numbers in $[-1,1]$. In \cite{LJH}, the authors used the noisy data in interior collocation points and boundary collocation points generated by the Gaussian random number.  The paper \cite{r141,LJH} is of NTSP style with very new interested numerical methods. The NTSP results suggested us to write future papers  devoted to the theoretical error estimates of the schemes.

The paper related to the TTSP style of the backward problem is 
quite scarce. In our knowledge, we can list here some related papers.  In Cavalier L. \cite{r1}, author gave some theoretical examples about inverse problems with random noise. Mair B. and Ruymgaart F. H. \cite{r7} considered theoretical formulas  for statistical inverse estimation in Hilbert scales
and applied the method for some examples. 
Our paper is inspired from the paper by Bissantz. N. and Holzmann. H. \cite{r3}
in which the authors considered a one--dimensional homogeneous backward problem. The very last papers are dealt with
i.i.d. random noises. In the present paper, we consider the nonhomogenous backward problem, which is of TTSP style with general non--i.i.d. noises and random sources. In our opinion, it is a positive point of our paper.\\

In the present paper,  we use the trigonometric method in nonparametric regression associated with the truncated expansion method to construct estimators which recover stably the Fourier coefficients of the unknown function $\theta(x,y)$. This ``hybrid" approach can be seen as a generalization of the one in \cite{r3} to the multi--dimensional and nonhomogeneous problem. Moreover, in \cite{r3}, an estimate of discretization bias of one--dimensional Fourier coefficients, one of the main part of the method, is stated heuristically without proof. Meanwhile, we are looking for an estimate so that it can be applied to the Sobolev class of functions.  To fill this gap in the two--dimensional case, we have to find a representation of the discretization bias by high--frequency Fourier coefficients of $h(x,y), f(x,y,t)$.\\

The rest of the paper is divided into 4 parts. In Section 2, we introduce the discretization form of Fourier coefficients.
Section 3 is devoted to the ill--posedness of the problem. In Section 4, we construct  estimator $\hat{\theta}(x,y)$
 for the initial temperature. We also give an upper bound for the error of estimation. Finally, we present some numerical results in Section 5. 

\section{Discretization form of Fourier coefficients}
In this paper, we denote
\begin{eqnarray}
L^2(\Omega) &=& \Bigg\{g:\Omega \to \mathbb{R} : \text{ $g$ is Lebesgue measurable and }\int_{\Omega} g^2(x,y) \mathrm{d} x \mathrm{d} y < \infty \Bigg\},\nonumber
\end{eqnarray}
with the inner product
\[\langle g_1, g_2\rangle = \int_{\Omega} g_1(x,y) g_2(x,y)\mathrm{d} x \mathrm{d} y,\] 
and the norm 
\[\|g\| = \sqrt{\int_{\Omega} g^2(x,y) \mathrm{d} x \mathrm{d} y}.\]
Here, we recall that $\Omega=(0,\pi)\times(0,\pi)$. For $p,q=1,2,\dots$, we put $\phi_p(x)=\sqrt{\dfrac{2}{\pi}}\sin px$ and $\phi_{p,q}(x,y)=\phi_p(x)\phi_q(y)$. As known, the system $\{\phi_{p,q}\}$ is completely orthonormal. Therefore
\[u(x,y,t)=\sum_{p=1}^{\infty}\sum_{q=1}^{\infty}u_{p,q}(t)\phi_{p,q}(x,y),\] 
where $u_{p,q}(t)=\langle u(\cdot,\cdot,t),\phi_{p,q}\rangle$. Similarly, we put
\begin{eqnarray}
\begin{array}{r l r l}
\theta_{p,q} &= \langle\theta,\phi_{p,q}\rangle, & \qquad f_{p,q}(t) &= \langle f(\cdot,\cdot,t),\phi_{p,q}\rangle , \\[8pt]
\lambda_{p,q}(t) &= e^{-A(t)\left(p^2 + q^2\right)}, & \qquad A(t) &= \displaystyle{\int_0^t a(\tau)\mathrm{d}\tau}.
\end{array}
\nonumber
\end{eqnarray}
Substituting the expansion of the function $u(x,y,t)$ into \eqref{prob2s} we obtain
\[\dfrac{\partial}{\partial t}u_{p,q}(t)+a(t)(p^2+q^2)u_{p,q}(t)=f_{p,q}(t).\] 
Thus
\[\dfrac{\partial}{\partial t}\left(e^{(p^2+q^2)\int_{0}^{t}a(\tau)\text{d}\tau}u_{p,q}(t)\right)=e^{(p^2+q^2)\int_{0}^{t}a(\tau)\text{d}\tau}f_{p,q}(t).\]
Solving this differential equation gives
\[u_{p,q}(t)=\left(\theta_{p,q} + \int_0^t \lambda^{-1}_{p,q}(\tau)f_{p,q}(\tau)\mathrm{d}\tau\right)\lambda_{p,q}(t).\] 
Hence, 
\begin{equation}
\label{solution1}
u(x,y,t) = \sum_{p=1}^\infty\sum_{q=1}^\infty \left(\theta_{p,q} + \int_0^t \lambda^{-1}_{p,q}(\tau)f_{p,q}(\tau)\mathrm{d}\tau\right)\lambda_{p,q}(t)\phi_{p,q}(x,y).
\end{equation}
Noting that
\begin{equation*}\label{theta0}
\theta(x,y) = u(x,y,0) = \sum_{p=1}^\infty\sum_{q=1}^\infty \theta_{p,q}\phi_{p,q}(x,y),
\end{equation*}
we can obtain the expansion
\begin{eqnarray*}
h(x,y) = u(x,y,T) &=& \sum_{p=1}^\infty\sum_{q=1}^\infty \left(\theta_{p,q} + \int_0^T \lambda^{-1}_{p,q}(\tau)f_{p,q}(\tau)\mathrm{d}\tau\right)\lambda_{p,q}(T)\phi_{p,q}(x,y).
\end{eqnarray*}
It follows that
\begin{equation}
h_{p,q} = \left(\theta_{p,q} + \int_0^T \lambda^{-1}_{p,q}(\tau)f_{p,q}(\tau)\mathrm{d}\tau\right)\lambda_{p,q}(T).
\label{h_pq}
\end{equation}

To establish an estimator for $\theta$, we need to recovery the Fourier coefficients $\theta_{p,q}, h_{p,q}$ and  $f_{p,q}(t)$ from $g_{ij}(t),d_{ij}$. Hence, we use approximation formulae of the coefficients which are constructed from the data--set. 
Suggested by one--dimensional estimators in \cite{r7}, \cite{r3}, 
we can construct a two--dimensional formula which give a discretization expansion for the Fourier coefficient $h_{p,q}$.
In fact, we claim that 
$$ h_{p,q}\approx\dfrac{\pi^2}{nm}\sum_{i=1}^n\sum_{j=1}^mh(x_i,y_j)\phi_{p,q}(x_i,y_j). $$
As mentioned in \cite{r3}, the discretization bias 
\begin{equation}
\gamma_{n,m,p,q}:= \dfrac{\pi^2}{nm}\sum_{i=1}^n\sum_{j=1}^mh(x_i,y_j)\phi_{p,q}(x_i,y_j) - h_{p,q}
\label{h_p,q}
\end{equation}
 is difficult to handle. 
In \cite{r3}, for brevity, the authors only assumed that the one--dimensional bias is of order $O(n^{-1})$. In the present paper, we shall give an explicitly estimate for the two--dimensional bias.

In fact, the formulae for the discretization bias will be derived from 
\begin{lemma}\label{lm1} Put 
$$ \delta_{p,q,r,s}= \dfrac{1}{n}\sum_{i=1}^{n}\phi_{p}(x_i)\phi_{r}(x_i) \dfrac{1}{m}\sum_{j=1}^{m} \phi_{q}(y_j)\phi_{s}(y_j). $$
For $p=\overline{1,n-1}$ and $q = \overline{1,m-1}$, with $x_i=\dfrac{\pi (2i-1)}{2n},y_j = \dfrac{\pi (2j-1)}{2m}$, we have
\begin{equation*}\label{tc1}
\delta_{p,q,r,s} = \left\{\begin{array}{*{20}{l}}
\hfill \dfrac{1}{\pi^2}, & (r,s)\pm (p,q) = (2kn,2lm),\\[8pt]
\hfill -\dfrac{1}{\pi^2}, & (r,s)\pm (-p,q) = (2kn,2lm),\\ [8pt]
\hfill 0, & \text{otherwise.}
\end{array} \right.
\end{equation*}
If $r = \overline{1, n-1}$ and $s = \overline{1,m-1}$, we obtain
\begin{equation*}\label{tc0}
\delta_{p,q,r,s} = \left\{\begin{array}{*{20}{l}}
\hfill \dfrac{1}{\pi^2}, & \qquad r = p \text { and } s = q,\\[8pt]
\hfill 0, & \qquad r \ne p \text{ or } s \ne q.
\end{array}\right.
\end{equation*}
\end{lemma}
\begin{proof}
The lemma is a direct consequence of Lemma 3.5 in \cite{r4}.
\end{proof}
From the latter lemma, we can represent the discretization bias $\gamma_{n,m,p,q}$ by high--frequency Fourier coefficients of the function $h$.  
Precisely, we have
\begin{lemma}\label{lm_hpq}
Assume that $h \in C^1(\overline{\Omega})$. For $p=\overline{1,n-1}$, $q=\overline{1,m-1}$. 
Then
\begin{equation}
\gamma_{n,m,p,q} = P_{n,p,q} + Q_{m,p,q} + R_{n,m,p,q}
\label{gamma0},
\end{equation}
with
\begin{eqnarray}
P_{n,p,q} &=& \sum_{k=1}^{\infty}(-1)^k h_{2kn \pm p,q},\qquad Q_{m,p,q} = \sum_{l=1}^{\infty}(-1)^l h_{p,2lm \pm q},\nonumber\\
R_{n,m,p,q} &=& \sum_{k=1}^{\infty}\sum_{l=1}^{\infty}(-1)^{k+1}\left(h_{2kn \pm p,2lm - q} + h_{2kn \pm p,2lm + q}\right).\nonumber
\end{eqnarray}
\end{lemma}
\begin{proof}
We have the following transform
\begin{eqnarray}
\dfrac{1}{m}\sum_{j=1}^m h(x_i,y_j)\phi_{q}(y_j) &=& \dfrac{1}{m}\sum_{j=1}^m \left(\sum_{r=1}^\infty \sum_{s=1}^\infty h_{r,s}\phi_r(x_i)\phi_s(y_j)\right)\phi_q(y_j)= \dfrac{1}{\pi}\sum_{r=1}^\infty h_{r,q}\phi_r(x_i) + S_{q}, \nonumber
\end{eqnarray}
where \[S_{q} = \dfrac{1}{\pi}\sum_{r=1}^{\infty} \phi_r(x_i)\sum_{l=1}^{\infty}(-1)^l h_{r,2lm \pm q}.\]
It follows that
\begin{eqnarray}
\dfrac{1}{n}\sum_{i=1}^{n}\left(\dfrac{1}{m}\sum_{j=1}^m h(x_i,y_j)\phi_q(y_j)\right)\phi_p(x_i) &=& \dfrac{1}{n}\sum_{i=1}^{n}\left(\dfrac{1}{\pi}\sum_{r=1}^\infty h_{r,q}\phi_r(x_i)\right)\phi_p(x_i) +  \dfrac{1}{n}\sum_{i=1}^{n}S_{q}\phi_p(x_i)\nonumber\\
&=& \dfrac{1}{\pi^2}(h_{p,q} + P_{n,p,q} + Q_{m,p,q} + R_{n,m,p,q}).\nonumber
\end{eqnarray} 
So the equality \eqref{gamma0} holds. 
\end{proof}
Now, we consider the discretization bias of Fourier coefficient $f_{p,q}(t)$ of the function $f(x,y,t)$ from the data--set. For convenient, we recall that
\begin{eqnarray}
f_{p,q}(t) &=& \langle f(\cdot,\cdot,t), \phi_{p,q}\rangle,\nonumber\\
f(x,y,t) &=& \sum_{p=1}^{\infty}\sum_{q=1}^{\infty}f_{p,q}(t)\phi_{p,q}(x,y).\nonumber
\end{eqnarray}
As in Lemma \ref{lm_hpq}, we can get similarly
\begin{lemma}
\label{lm_f_{pq}}
Assume that $f \in C([0,T];C^1(\overline{\Omega}))$, $p=\overline{1,n-1}$ and $q = \overline{1,m-1}$. Put
\begin{equation}
\eta_{n,m,p,q}(t) = \dfrac{\pi^2}{nm}\sum_{i=1}^{n}\sum_{j=1}^{m}f(x_i,y_j,t)\phi_{p,q}(x_i,y_j) - f_{p,q}(t).
\label{f_pq}
\end{equation}
Then
\begin{equation}
\eta_{n,m,p,q}(t) = P'_{n,p,q}(t) + Q'_{m,p,q}(t) + R'_{n,m,p,q}(t),
\label{eta0}
\end{equation}
with
\begin{eqnarray}
P'_{n,p,q}(t) &=&\sum_{k=1}^{\infty}(-1)^k f_{2kn \pm p,q}(t), \qquad Q'_{m,p,q}(t) = \sum_{l=1}^{\infty}(-1)^l f_{p,2lm \pm q}(t), \nonumber\\
R'_{n,m,p,q}(t) &=&\sum_{k=1}^\infty\sum_{l=1}^\infty(-1)^{l+k} (f_{2kn \pm p,2lm + q}(t)+f_{2kn \pm p,2lm - q}(t)).\nonumber
\end{eqnarray}
\end{lemma}
Combining equalities \eqref{h_pq}, \eqref{h_p,q} and \eqref{f_pq} we can obtain a data--explicit form for $\theta(x,y)$
\begin{theorem}\label{thr1}
Let $M,N \in \mathbb{N}$ such that $0<N\leq n$, $0<M\leq m$. Assume that the functions $h$, $f$ are fulfilled Lemma \ref{lm_hpq} and Lemma \ref{lm_f_{pq}} and that $u$ is as in \eqref{solution1}. Then
\begin{eqnarray*}\label{theta1}
\theta(x,y) &=& \sum_{p=1}^N \sum_{q=1}^M\Bigg[\dfrac{\pi^2}{nm}\sum_{i=1}^n\sum_{j=1}^m \left(h(x_i,y_j)\lambda_{p,q}^{-1}(T)-
\int_{0}^{T}\lambda_{p,q}^{-1}(\tau)f(x_i,y_j,\tau)\mathrm{d}\tau\right)\phi_{p,q}(x_i,y_j)\nonumber \\ && - \: \left(\gamma_{n,m,p,q}\lambda_{p,q}^{-1}(T)-\int_{0}^{T}\lambda_{p,q}^{-1}(\tau)\eta_{n,m,p,q}(\tau)\mathrm{d}\tau\right)\Bigg] \phi_{p,q}(x,y)\nonumber \\ && + \: \sum_{p=N+1}^\infty \sum_{q=1}^M \theta_{p,q}\phi_{p,q}(x,y) + \sum_{p=1}^N \sum_{q=M+1}^\infty \theta_{p,q}\phi_{p,q}(x,y) +  \sum_{p=N+1}^\infty \sum_{q=M+1}^\infty \theta_{p,q}\phi_{p,q}(x,y),
\end{eqnarray*}
where $\gamma_{n,m,p,q}, \eta_{n,m,p,q}$ are as in Lemma \ref{lm_hpq} and Lemma \ref{lm_f_{pq}}.
\end{theorem}
\section{The ill--posedness of the problem}
 From the theorem, we can consider the ill--posedness of our problem. We investigate a
concrete model of data and prove the instability of the solution in the case of random noise data.
Suppose that $h(x,y)=f(x,y,t)\equiv0$ and $a(t)=1$, $u(x,y,T)=0$. The unique solution of \eqref{prob2s}--\eqref{cd1} 
is $u(x,y,t)\equiv0$. 

Let the random noise data be
\begin{eqnarray*}
g_{ij}(t) &=& 0 + \vartheta\xi_{ij}(t),\\
d_{ij} &=& 0 + \epsilon_{ij},\qquad\epsilon_{ij}\mathop\sim\limits^{i.i.d}\mathcal{N}(0,n^{-1}m^{-1}),
\end{eqnarray*}
for $i=\overline{1,n},j=\overline{1,m}$.  We shall construct the solution 
of \eqref{prob2s}--\eqref{cd1} with respect to the random data.
 Using the idea of the trigonometric regression, we put 
\begin{eqnarray*}
\overline{h}^{nm}(x,y)&=&\sum_{p=1}^{n-1}\sum_{q=1}^{m-1}\overline{h}_{p,q}^{mn}\phi_{p,q}(x,y),\\
\overline{f}^{nm}(x,y,t)&=&\sum_{p=1}^{n-1}\sum_{q=1}^{m-1}\overline{f}_{p,q}^{mn}(t)\phi_{p,q}(x,y),
\end{eqnarray*}
where
\[\overline{h}_{p,q}^{nm}=\frac{\pi^2}{nm}\sum_{i=1}^{n}\sum_{j=1}^{m}\epsilon_{ij}\phi_{p,q}(x_i,y_j),\]
\[\overline{f}_{p,q}^{nm}(t)=\frac{\pi^2\vartheta}{nm}\sum_{i=1}^{n}\sum_{j=1}^{m}\xi_{ij}(t)\phi_{p,q}(x_i,y_j).\]
The definition implies
\[\overline{\gamma}_{n,m,p,q}:=\frac{\pi^2}{nm}\sum_{i=1}^{n}\sum_{j=1}^{m}\epsilon_{ij}\phi_{p,q}(x_i,y_j)-\overline{h}_{p,q}^{nm}=0,\]
\[\overline{\eta}_{n,m,p,q}(t):=\frac{\pi^2\vartheta}{nm}\sum_{i=1}^{n}\sum_{j=1}^{m}\xi_{ij}(t)\phi_{p,q}(x_i,y_j)
-\overline{f}_{p,q}^{nm}(t)=0.\]
By the orthogonal property stated in  Lemma \ref{tc1}, we can verify directly that
$$\overline{h}_{nm}(x_i,y_j)=d_{ij}, \overline{f}_{nm}(x_i,y_j,t)=g_{ij}(t).$$
 Let
$\overline{u}=\overline{u}(x,y,t)$ be the solution of the system 
\begin{equation*}
\label{prob2sr}
\left\{\begin{array}{l l l}
\overline{u}_t - (\overline{u}_{xx} + \overline{u}_{yy}) & = \overline{f}^{nm}(x,y,t), & \qquad (x,y,t) \in \Omega\times (0,T),\\
\overline{u}(x,y,T) & = \overline{h}^{nm}(x,y), & \qquad (x,y) \in\overline{\Omega},
\end{array}
\right.
\end{equation*}
subject to the Dirichlet condition
\begin{equation*}
\overline{u}(0,y,t) = \overline{u}(\pi,y,t) = \overline{u}(x,0,t) = \overline{u}(x,\pi,t) = 0.
\label{cd1r}
\end{equation*}
We can remark  that $\overline{u}(\cdot,\cdot,t)$ is a trigonometric polynomial with order $<n$ (with respect to the variable $x$) and order $<m$ (with respect to the variable $y$). Putting $\overline{\theta}^{nm}(x,y)=\overline{u}(x,y,0)$, we get in view of the remark that
$$\overline{\theta}_{p,q}^{nm}:=\langle \overline{\theta}^{nm},\phi_{p,q}\rangle,$$ 
for $p\ge n$ or $q\ge m$.

Applying Theorem \ref{thr1} with $N=n-1,M=m-1$, we obtain 
\[\overline{\theta}^{nm}(x,y)=\sum_{p=1}^{n-1}\sum_{q=1}^{m-1}\left(\overline{h}_{p,q}^{nm}-\int_{0}^{T}\lambda_{p,q}^{-1}(\tau)\overline{f}_{p,q}^{nm}(\tau)\text{d}\tau\right)\lambda_{p,q}^{-1}(T)\phi_{p,q}(x,y),\]
thus
\begin{eqnarray*}
\|\overline{\theta}^{nm}\|^2&=&\sum_{p=1}^{n-1}\sum_{q=1}^{m-1}\left(\overline{h}_{p,q}^{nm}-\int_{0}^{T}\lambda_{p,q}^{-1}(\tau)\overline{f}_{p,q}^{nm}(\tau)\text{d}\tau\right)^2\lambda_{p,q}^{-2}(T)\\
&\ge& \left(\overline{h}_{n-1,m-1}^{nm}-\int_{0}^{T}\lambda_{n-1,m-1}^{-1}(\tau)\overline{f}_{n-1,m-1}^{nm}(\tau)\text{d}\tau\right)^2\lambda_{n-1,m-1}^{-2}(T).
\end{eqnarray*} 
Assuming that the random quantities $\epsilon_{ij}$ and $\xi_{ij}(t)$ are mutually independent, we can obtain by direct computation that 
$$\lim_{n,m\to\infty}\mathbb{E}\|\overline{f}^{n,m}(\cdot,\cdot,t)\|^2=0,\qquad\forall t\in[0,T].$$
Moreover, by the Parseval equality, we have 
\[\|\overline{h}^{nm}\|^2=\sum_{p=1}^{n-1}\sum_{q=1}^{m-1}\left(\overline{h}_{p,q}^{mn}\right)^2=\sum_{p=1}^{n-1}\sum_{q=1}^{m-1}\frac{\pi^4}{n^2m^2}\left(\sum_{i=1}^{n}\sum_{j=1}^{m}\epsilon_{ij}\phi_{p,q}(x_i,y_j)\right)^2.\]
Using Lemma \ref{lm1}, we obtain
\[\mathbb{E}\|\overline{h}^{nm}\|^2=\sum_{p=1}^{n-1}\sum_{q=1}^{m-1}\frac{\pi^2}{nm}\sum_{i=1}^{n}\sum_{j=1}^{m}\mathbb{E}\epsilon_{ij}^2=\frac{(n-1)(m-1)}{n^2m^2}.\] Thus \[\lim_{n,m\to\infty}\mathbb{E}\|\overline{h}^{nm}\|^2=0.\]
On the other hand, we claim that $\mathbb{E}\|\overline{\theta}_{nm}\|^2\to\infty$ as $n,m\to\infty$. In fact, we have
\begin{eqnarray*}
\mathbb{E}\|\overline{\theta}_{nm}\|^2 &\ge& \left[\mathbb{E}(\overline{h}_{n,m}^{nm})^2+\mathbb{E}\left(\int_{0}^{T}e^{-\tau(n^2+m^2)}\overline{f}_{p,q}^{nm}(\tau)\text{d}\tau\right)^2\right]e^{2T(n^2+m^2)} \ge \frac{\pi^2}{n^2m^2}e^{2T(n^2+m^2)}
\end{eqnarray*} and \[\mathbb{E}\|\overline{\theta}_{nm}\|^2\rightarrow+\infty \text{ as in } n,m\rightarrow+\infty.\]
From the latter inequality, we can deduce that the problem is ill--posed. Moreover, as classified in \cite{r1}, the problem is severely ill--posed. Hence, a regularization is in order.

\section{Estimators and Convergence results}

In Section 3, we have known that the problem is ill--posed. To deal with it, we have some regularization methods, for instance, one can employ the quasi--boundary value method (QBV) \cite{r18}; or use the Tikhonov method.... In this paper, we use the truncated method in analogy to 1--dimension problem of \cite{r3}. The advantage of this method is that it seems to be more convenient for computation, because we can control stopping criterion.

Let two natural numbers $N$ and $M$ be the regularization parameters. To construct an estimator $\hat{\theta}_{n,m,N,M}$, we first note that the quantities $\gamma_{n,m,p,q}$ and $\eta_{n,m,p,q}$ are small when $n$ and $m$ go to infinity (see Lemmas \ref{norm_h} and \ref{eta}). This leads to
\begin{eqnarray}
\hat{h}_{p,q} &=& \frac{\pi^2}{nm}\sum_{i=1}^n\sum_{j=1}^m d_{ij}\phi_{p,q}(x_i,y_j) \nonumber,\\
\hat{f}_{p,q}(t) &=& \frac{\pi^2}{nm}\sum_{i=1}^{n}\sum_{j=1}^{m}g_{ij}(t)\phi_{p,q}(x_i,y_j), \nonumber 
\end{eqnarray}
where $p = \overline{1,n-1}, q = \overline{1,m-1}$. Based on that and equation \eqref{h_pq}, $\hat{\theta}_{n,m,p,q}$ may be defined as
\begin{equation*}
\hat{\theta}_{n,m,q,p} = \hat{h}_{p,q}\lambda_{p,q}^{-1}(T) - \int_0^T \lambda_{p,q}^{-1}(\tau)\hat{f}_{p,q}(\tau) \mathrm{d}\tau.
\end{equation*}
This, together with an application of the truncated expansion method to the formula of $\theta$ in Theorem \ref{thr1}, leads to
\begin{equation}\label{theta2}
\hat{\theta}_{n,m,N,M}(x,y) = \sum_{p=1}^N \sum_{q=1}^M \hat{\theta}_{n,m,q,p}\phi_{p,q}(x,y).
\end{equation}

Now, we study the convergence rate, which is the main result in this paper. Hereafter, for any positive numbers $\alpha, \beta$ and $E$, we denote the Sobolev class of functions by
\begin{equation*}\label{ass1}
\mathscr{C}_{\alpha,\beta,E} = \left\{g \in L^2(\Omega): \sum_{p=1}^\infty\sum_{q=1}^\infty p^{2\alpha}q^{2\beta}\left|\big\langle g,\phi_{p,q}\big\rangle\right|^2 \le E^2\right\}.
\end{equation*}
The convergence rate of estimator $\hat{\theta}_{n,m,N,M}(x,y)$ in \eqref{theta2} is presented by Theorem \ref{thr2}. In order to prove the theorem, we need the evaluation for $\mathbb{E}\left\|\hat{\theta}_{n,m,N,M} - \theta \right\|^2$. In fact, this estimate procedure has to undergo some important steps. In the first step, we have
\begin{lemma}\label{norm_theta0}
Let the regression models \eqref{st_md} and \eqref{md2} hold. Assume that $\theta \in \mathscr{C}_{\alpha,\beta,E}$ and $0<N<n,0<M<m$. Then
\begin{eqnarray}\label{nor0}
\lefteqn{\left\|\hat{\theta}_{n,m,N,M} - \theta \right\|^2}\nonumber\\ &=& 4\sum_{p=1}^N \sum_{q=1}^M\Bigg[\dfrac{\pi^2}{nm}\sum_{i=1}^n\sum_{j=1}^m\left(\lambda^{-1}_{p,q}(T)\sigma_{ij}\epsilon_{ij} - \vartheta\int_0^T\lambda^{-1}_{p,q}(\tau)\xi_{ij}(\tau)\mathrm{d}\tau\right)\phi_{p,q}(x_i,y_j)\nonumber\\ && - \: \int_0^T\lambda^{-1}_{p,q}(\tau)\eta_{n,m,p,q}(\tau)\mathrm{d}\tau + \gamma_{n,m,p,q}\lambda_{p,q}^{-1}(T)\Bigg]^2 \nonumber\\ && + \: 4\left(\sum_{p=N+1}^\infty \sum_{q=1}^M \theta^2_{p,q} + \sum_{p=1}^N \sum_{q=M+1}^\infty \theta_{p,q}^2 + \sum_{p=N+1}^\infty \sum_{q=M+1}^\infty \theta_{p,q}^2\right),
\end{eqnarray}
where we recall $\theta_{p,q} = \langle\theta, \phi_{p,q}\rangle$.
\end{lemma}
\begin{proof}
By the Parseval equality, we have
\begin{eqnarray}
\left\|\hat{\theta}_{n,m,N,M} - \theta \right\|^2 &=& 4\sum_{p=1}^N \sum_{q=1}^M \left(\hat{A}_{p,q} - \theta_{p,q}\right)^2 \nonumber\\ && +\: 4\left(\sum_{p=N+1}^\infty \sum_{q=1}^M \theta_{p,q}^2 + \sum_{p=1}^N \sum_{q=M+1}^\infty \theta_{p,q}^2 + \sum_{p=N+1}^\infty \sum_{q=M+1}^\infty \theta_{p,q}^2\right).\nonumber
\end{eqnarray}
From the formula of $\hat{A}_{p,q}$ and $\theta_{p,q},p=\overline{1,N},q=\overline{1,M},$ we get
\begin{eqnarray}
\hat{A}_{p,q} - \theta_{p,q} &=& \dfrac{\pi^2}{nm}\lambda_{p,q}^{-1}(T)\sum_{i=1}^{n}\sum_{j=1}^{m}\sigma_{ij}\epsilon_{ij}\phi_{p,q}(x_i,y_j) \nonumber\\ && - \: \int_0^T\lambda_{p,q}^{-1}(\tau)\left[\hat{f}_{p,q}(\tau) - f_{p,q}(\tau)\right]\mathrm{d}\tau - \gamma_{n,m,p,q}\lambda_{p,q}^{-1}(T),\nonumber
\end{eqnarray}
with
\begin{eqnarray}
\hat{f}_{p,q}(t) - f_{p,q}(t) &=& \dfrac{\pi^2\vartheta}{nm}\sum_{i=1}^{n}\sum_{j=1}^{m}\xi_{ij}(t)\phi_{p,q}(x_i,y_j) + \eta_{n,m,p,q}(t).\nonumber
\end{eqnarray}
Thus, we obtain~\eqref{nor0}.
\end{proof}
Now, we prove that $\gamma_{n,m,p,q}$ and $\eta_{n,m,p,q}$ are ``small" in an appropriate sense. We first have
\begin{lemma}
\label{norm_h}
Assume that $f(\cdot,\cdot,t) \in \mathscr{C}_{\alpha,\beta,E}$ for all $t \in [0,T]$ and  $\theta,h\in L^2(\Omega)$. Then
\begin{equation}
|h_{p,q}| \le \|\theta\|\lambda_{p,q}(T) + \dfrac{E}{p^\alpha q^\beta a_1(p^2+q^2)}.
\label{fourier_h1}
\end{equation}
\end{lemma}
\begin{proof}
From \eqref{h_pq} and $|f_{p,q}(\cdot)|\le E/(p^\alpha q^\beta)$, we have
\begin{eqnarray}
|h_{p,q}| &\le& \left(|\theta_{p,q}| + \int_0^T \lambda^{-1}_{p,q}(\tau)|f_{p,q}(\tau)|\mathrm{d}\tau\right)\lambda_{p,q}(T)\nonumber\\
&\le& \left(\|\theta\| +  \dfrac{E}{p^\alpha q^\beta}\int_0^T \lambda^{-1}_{p,q}(\tau)\mathrm{d}\tau\right)\lambda_{p,q}(T)\nonumber\\
&\le& \|\theta\|\lambda_{p,q}(T) + \dfrac{E}{p^\alpha q^\beta}\int_0^T e^{-(p^2+q^2)\int_{\tau}^{T}a(s)\mathrm{d}s} \mathrm{d}\tau. \nonumber
\end{eqnarray}
Since $a(t)\geq a_1$, we deduce
\begin{eqnarray}
|h_{p,q}|&\le& \|\theta\| \lambda_{p,q}(T) + \dfrac{E}{p^\alpha q^\beta} \int_{0}^{T} e^{a_1(\tau - T)(p^2+q^2)} \mathrm{d}\tau\nonumber\\
&\le& \|\theta\| \lambda_{p,q}(T) + E\dfrac{1 - e^{-a_1T(p^2+q^2)}}{p^\alpha q^\beta a_1(p^2+q^2)}\le \|\theta\| \lambda_{p,q}(T) + \dfrac{E}{p^\alpha q^\beta a_1(p^2+q^2)}.\nonumber
\end{eqnarray}
This completes the proof.
\end{proof}
Now, in the next lemma we shall give an upper bound for the discretization bias of $h_{p,q}$. In fact, we have
\begin{lemma}
\label{gamma}
Suppose that $f(\cdot,\cdot,t) \in \mathscr{C}_{\alpha,\beta,E}$ and that $p=\overline{1,n-1},q=\overline{1,m-1}$. With $\gamma_{n,m,p,q}$ defined by~\eqref{h_p,q}, there is a generic constant $C$ independent of
$n,m,p,q$ such that
\begin{equation}
|\gamma_{n,m,p,q}| \le 
 Cn^{-1-\alpha/2}m^{-1-\beta/2}.
\label{gamma1_0}
\end{equation}
\end{lemma}
\begin{proof}\text{}\\
From \eqref{gamma0}, we have
\begin{equation*}
|\gamma_{n,m,p,q}| \le |P_{n,p,q}| + |Q_{m,p,q}| + |R_{n,m,p,q}|.
\label{gamma1}
\end{equation*}
Using Lemma \ref{norm_h} gives
\begin{eqnarray*}
|P_{n,p,q}| &\le& \sum_{k=1}^\infty |h_{2kn \pm p,q}|\nonumber\\
 &\le& \|\theta\|\sum_{k=1}^\infty \lambda_{2kn \pm p,q}(T) + \sum_{k=1}^\infty\dfrac{E}{(2kn \pm p)^\alpha q^\beta a_1((2kn \pm p)^2+q^2)}\nonumber\\
&\le& \|\theta\|\sum_{k=1}^{\infty} e^{-A(T)[(2kn \pm p)^2 + q^2]} + \sum_{k=1}^\infty \dfrac{E}{a_1[(2kn \pm p)^{2+\alpha}+q^{2+\beta}]}.\nonumber
\end{eqnarray*}
This follows that
\begin{eqnarray*}
|P_{n,p,q}|&\le& \|\theta\| e^{-A(T)q^2}\sum_{k=1}^{\infty} e^{-A(T)(2kn \pm p)} + \sum_{k=1}^{\infty}\dfrac{E}{a_1(2kn \pm p)^{2+\alpha}}\nonumber\\
&\le& \|\theta\|\dfrac{e^{-A(T)(2n-p+q^2)} + e^{-A(T)(2n+p+q^2)}}{1-e^{-2nA(T)}} + \sum_{k=1}^{\infty}\dfrac{E}{a_1(2kn \pm n)^{2+\alpha}}\nonumber\\
&\le& \|\theta\|\dfrac{2e^{-A(T)(2n-p+q^2)}}{1-e^{-2nA(T)}} + \dfrac{E}{a_1 n^{2+\alpha}}\sum_{k=1}^{\infty}\dfrac{1}{(2k \pm 1)^{2+\alpha}}.\nonumber\\
\end{eqnarray*}
Since $A(T) > a_1T$ and $1-e^{-2nA(T)}\ge\dfrac{1}{2}$ as $n$ large, we obtain
\begin{equation}
\sum_{k=1}^\infty |h_{2kn \pm p,q}| \le 4e^{-a_1 T(2n-p+q^2)}\|\theta\| + \dfrac{2E K_\alpha}{a_1 n^{2+\alpha}} := K_{1,n,m},
\label{gamma1_1}
\end{equation}
where we use $K_\alpha:=\sum_{k=1}^{\infty}\dfrac{1}{(2k-1)^{2+\alpha}}<2,\forall\alpha>0$.
Similarly, we get
\begin{equation}
|Q_{m,p,q}| \le 4e^{-a_1T(2m-q+p^2)}\|\theta\| + \dfrac{2E K_\beta}{a_1m^{2+\beta}} := K_{2,n,m}.
\label{gamma1_2}
\end{equation}
Next, we find an upper bound for $|R_{n,m,p,q}|$. In fact, we have
\begin{eqnarray}
|R_{n,m,p,q}| &\le& \sum_{k=1}^{\infty}\sum_{l=1}^{\infty}|h_{2kn \pm p,2lm - q}| + \sum_{k=1}^{\infty}\sum_{l=1}^{\infty}|h_{2kn \pm p,2lm + q}|.\nonumber
\end{eqnarray}
Now we estimate the first term as follows
\begin{eqnarray}
\lefteqn{\sum_{k=1}^{\infty}\sum_{l=1}^{\infty}|h_{2kn \pm p,2lm - q}|} \nonumber\\
&\le& 
\sum_{k=1}^{\infty}\sum_{l=1}^{\infty} \|\theta\|\lambda_{2kn \pm p,2lm - q}(T) + \sum_{k=1}^{\infty}\sum_{l=1}^{\infty}\dfrac{E}{(2kn \pm p)^\alpha (2lm-q)^\beta a_1((2kn \pm p)^2+(2lm - q)^2)}\nonumber\\
&\le& \|\theta\|\sum_{k=1}^{\infty}\sum_{l=1}^{\infty}e^{-A(T)[(2kn \pm p)^2 + (2lm-q)^2]} + \sum_{k=1}^{\infty}\sum_{l=1}^{\infty}\dfrac{E}{a_1\left[(2kn \pm p)^{2+\alpha} + (2lm-q)^{2+\beta}\right]}\nonumber\\
&\le&\dfrac{\|\theta\|\left(e^{-A(T)(2n+2m-p-q)} + e^{-A(T)(2n+2m+p-q)}\right)}{\left[1-e^{-2nA(T)}\right]\left[1-e^{-2mA(T)}\right]} + \sum_{k=1}^{\infty}\sum_{l=1}^{\infty}\dfrac{E}{a_1\left[(2kn \pm n)^{2+\alpha} + 
(2lm-m)^{2+\beta}\right]}.\nonumber
\end{eqnarray}
Using the inequality $x+y\geq 2\sqrt{xy}$ ($x,y\geq 0$), we obtain
\begin{eqnarray}
\lefteqn{\sum_{k=1}^{\infty}\sum_{l=1}^{\infty}|h_{2kn \pm p,2lm - q}|} \nonumber\\
&\le&  \|\theta\|\dfrac{2e^{-A(T)(2n+2m-p-q)}}{\left[1-e^{-2nA(T)}\right]\left[1-e^{-2mA(T)}\right]} + \dfrac{E}{2a_1n^{1+\alpha/2}m^{1+\beta/2}}\sum_{k=1}^\infty \sum_{l=1}^\infty\dfrac{1}{(2k-1)^{1+\alpha/2}(2l-1)^{1+\beta/2}} \nonumber\\
&\le& 8e^{-a_1T(2n+2m-p-q)}\|\theta\| + \dfrac{E K_{\alpha,\beta}}{2a_1n^{1+\alpha/2}m^{1+\beta/2}},\nonumber
\end{eqnarray}
where $K_{\alpha,\beta}:=\sum_{k=1}^\infty \sum_{l=1}^\infty\dfrac{1}{(2k-1)^{1+\alpha/2}(2l-1)^{1+\beta/2}}<+\infty,\forall\alpha,\beta>0$. Similarly, we get
$$ \sum_{k=1}^{\infty}\sum_{l=1}^{\infty}|h_{2kn \pm p,2lm + q}|
\le 8e^{-a_1T(2n+2m-p-q)}\|\theta\| + \dfrac{E K_{\alpha,\beta}}{2a_1n^{1+\alpha/2}m^{1+\beta/2}}. $$
Therefore
\begin{equation*}
|R_{n,m,p,q}| \le 16e^{-a_1T(2n+2m-p-q)}\|\theta\| + \dfrac{E K_{\alpha,\beta}}{a_1n^{1+\alpha/2}m^{1+\beta/2}} := K_{3,n,m}.
\label{gamma1_3}
\end{equation*}
Noting that $ 2(K_{1,n,m} + K_{2,n,m})\le Cn^{-1-\alpha/2}m^{-1-\beta/2}$ and that $K_{3,n,m}\leq O(n^{-1-\alpha/2}m^{-1-\beta/2})$, we get the inequality \eqref{gamma1_0}.
\end{proof}
\noindent{\bf Remark.}
Writing almost verbatim (in fact, easier) the above proof, we can obtain an estimation of order $O(n^{-1-\alpha/2})$ for the
 discretization bias of one--dimensional Fourier coefficients. The order  
is better than the order $O(n^{-1})$ assumed in \cite{r3} and it can be applied for the Sobolev class of functions . Moreover, the idea can be generalized to the $n$--dimensional case.
\begin{lemma}
\label{eta}
Assume that $f(\cdot,\cdot,t) \in \mathscr{C}_{\alpha,\beta,E}$ and $\alpha,\beta>1$. With $\eta_{n,m,p,q}(t)$ defined by~\eqref{f_pq}, we obtain
\begin{equation}
|\eta_{n,m,p,q}(t)| \le C'\left(n^{-\alpha}+m^{-\beta}\right), 
\label{eta1}
\end{equation}
where  $2\le C' < \infty$.
\end{lemma}
\begin{proof}\text{}\\
From \eqref{eta0}, the triangle inequality implies
\[|\eta_{n,m,p,q}(t)| \le |P'_{n,p,q}(t)| + |Q'_{m,p,q}(t)| + |R'_{n,m,p,q}(t)|.\]
Estimating directly the first term gives
\begin{eqnarray}
|P'_{n,p,q}(t)| &\le& \sum_{k=1}^{\infty}\left|f_{-p+2kn,q}(t)\right|+\left|f_{p+2kn,q}(t)\right| \le E\sum_{k=1}^{\infty}\left(\dfrac{1}{(2kn-p)^\alpha q^\beta}+\dfrac{1}{(p+2kn)^\alpha q^\beta}\right) \nonumber\\
&\le& \sum_{k=1}^\infty \dfrac{2E}{(2kn-p)^\alpha} \le 2\sum_{k=1}^{\infty}\dfrac{2E}{(2kn-n)^\alpha} \le \dfrac{C_\alpha}{n^{\alpha}}.\nonumber
\end{eqnarray}
Similarly, we also have
\begin{eqnarray}
|Q'_{m,p,q}(t)| &\le& \sum_{l=1}^{\infty}\left|f_{p,-q+2lm}(t)+f_{p,q+2lm}(t)\right|\le \dfrac{C_\beta}{m^{\beta}}\nonumber
\end{eqnarray}
and
\begin{eqnarray}
|R'_{n,m,p,q}(t)| &\le& \sum_{k=1}^{\infty}\sum_{l=1}^{\infty} \left|f_{2kn \pm p,2lm-q}(t)+f_{2kn \pm p,2lm+q}(t)\right|\nonumber\\
&\le& \dfrac{C_{\alpha,\beta}}{n^\alpha m^\beta}\nonumber
\end{eqnarray}
with $4\le C_{\alpha,\beta} < \infty$. Moreover, we easily see that the upper bound of $|R'_{n,m,p,q}(t)|$ is very smaller than the upper bounds of $|P'_{n,p,q}(t)|$ and $|Q'_{m,p,q}(t)|$ as $n,m$ tend to infinity. Hence, we get \eqref{eta1}.
\end{proof}
To prepare for the proof of the main result, we need
\begin{lemma}\label{appro_int1}
Let $L > 1$ and $k > 0$, then
\begin{equation}
\label{appx1}
\int_{1}^{L} e^{ku^2} \mathrm{d}u \le \frac{1}{Lk}e^{L^2k}.
\end{equation}
\end{lemma}
\begin{proof}\text{}\\
Putting $s = \dfrac{u}{L}$, we have
\[\int_{1}^{L}e^{ku^2} \mathrm{d}u = L\int_{\frac{1}{L}}^{1}e^{L^2k s^2} \mathrm{d}s \le L\int_{0}^{1}e^{L^2k s^2} \mathrm{d}s.\]
Then, transforming variable $v=L^2k(1-s)$ gives
\begin{eqnarray}
L\int_{0}^{1}e^{L^2k s^2} \mathrm{d}s = \dfrac{1}{Lk}\int_{0}^{L^2k}e^{L^2k\left(1 - \frac{v}{L^2k}\right)^2} \mathrm{d} v =  \dfrac{1}{Lk}e^{L^2k}\int_0^{L^2k}e^{L^2k\left(\left(1 - \frac{v}{L^2k}\right)^2 - 1\right)} \mathrm{d} v.\nonumber
\end{eqnarray}
Since
\[L^2k\left(\left(1 - \frac{v}{L^2k}\right)^2 - 1\right) = v\dfrac{L^2k\left(\left(1 - \frac{v}{L^2k}\right)^2 - 1\right)}{v} \leq - v,\]
we have
\begin{eqnarray}
\int_{1}^{L} e^{ku^2} \mathrm{d}u &\le&  \dfrac{1}{Lk}e^{L^2k}\int_0^{L^2k}e^{-v} \mathrm{d} v \le  \dfrac{1}{Lk}e^{L^2k} \left(1 - e^{-L^2k}\right) \le  \dfrac{1}{Lk}e^{L^2k}. \nonumber
\end{eqnarray}
Therefore, \eqref{appx1} holds.
\end{proof}
Finally, we are ready to state and prove the main theorem of our paper.
\begin{theorem}\label{thr2}
Let $E > 0$, $\alpha,\beta > 1$, $0<\omega_1,\omega_2<2$ and
$h \in C^1(\overline{\Omega}), \quad f \in C([0,T];C^1(\overline{\Omega})\cap \mathscr{C}_{\alpha,\beta,E})$.
Assume that the system \eqref{prob2s}--\eqref{cd1} has a (unique) solution
$u \in C^1([0,1];L^2(\Omega)) \cap C([0,T];H^2(\Omega))$.
Choose
\begin{equation*}\label{choos_M,N}
N = \left\lfloor\dfrac{(\omega_1\log n)^{1/2}}{2\sqrt{A(T)}}\right\rfloor \text{ and } M = \left\lfloor\dfrac{(\omega_2\log m)^{1/2}}{2\sqrt{A(T)}}\right\rfloor,
\end{equation*}
where  $\lfloor x\rfloor$ is the greatest integer $\leq x$.
For $\hat{\theta}_{n,m,N,M}(x,y)$ defined in \eqref{theta2}, $\theta(x,y)=u(x,y,0)$,  we have
\begin{equation*}\label{nor1}
\mathbb{E}\left\|\hat{\theta}_{n,m,N,M} - \theta \right\|^2 \le C_0\left(\left(\dfrac{\omega_1
\log n}{4A(T)}\right)^{-\alpha} + \left(\dfrac{\omega_2
\log m}{4A(T)}\right)^{-\beta}\right).
\end{equation*}
Here, the positive constant $C_0$ is  independent of $n,m$.
\end{theorem}
\begin{proof}\text{}\\
According Lemma~\ref{norm_theta0}, we have
\begin{equation*}
\mathbb{E}\left\|\hat{\theta}_{n,m,N,M} - \theta \right\|^2 \le  \mathbb{E}I_1 + I_2,
\label{norm_theta1}
\end{equation*}
where
\begin{eqnarray}
I_1 &=& \dfrac{12\pi^4}{n^2m^2}\sum_{p=1}^N \sum_{q=1}^M\Bigg[\left(\sum_{i=1}^n\sum_{j=1}^m\left(\lambda^{-1}_{p,q}(A(T))\sigma_{ij}\epsilon_{ij} - \int_0^T\lambda^{-1}_{p,q}(\tau)\vartheta\xi_{ij}(\tau)\mathrm{d}\tau\right)\phi_{p,q}(x_i,y_j)\right)^2\nonumber\\ && + \: \left(\int_0^T\lambda^{-1}_{p,q}(\tau)\eta_{n,m,p,q}(\tau)\mathrm{d}\tau\right)^2 + \gamma_{n,m,p,q}^2\lambda_{p,q}^{-2}(A(T))\Bigg],\nonumber\\
I_2 &=& 4\left(\sum_{p=N+1}^\infty \sum_{q=1}^M \theta^2_{p,q} + \sum_{p=1}^N \sum_{q=M+1}^\infty \theta_{p,q}^2 + \sum_{p=N+1}^\infty \sum_{q=M+1}^\infty \theta_{p,q}^2\right).\nonumber
\end{eqnarray}
First, we consider $I_1$. We have
\[I_1 = \dfrac{12\pi^4}{n^2m^2}\left(I_{1,1} + I_{1,2} + I_{1,3}\right).\]
We get
\begin{eqnarray}
I_{1,1}&=&\sum_{p=1}^N \sum_{q=1}^M\left(\sum_{i=1}^n\sum_{j=1}^m\left(\lambda^{-1}_{p,q}(A(T))\sigma_{ij}\epsilon_{ij} - \int_0^T\lambda^{-1}_{p,q}(\tau)\vartheta\xi_{ij}(\tau)\mathrm{d}\tau\right)\phi_{p,q}(x_i,y_j)\right)^2\nonumber\\
&\le& 2\sum_{p=1}^N \sum_{q=1}^M\Bigg(\lambda^{-2}_{p,q}(A(T))\left[\sum_{i=1}^n\sum_{j=1}^m\phi_{p,q}(x_i,y_j)\sigma_{ij}\epsilon_{ij}\right]^2 \nonumber\\ && +\: \left[\sum_{i=1}^n\sum_{j=1}^m\phi_{p,q}(x_i,y_j)\int_0^T\lambda^{-1}_{p,q}(\tau)\vartheta\xi_{ij}(\tau) \mathrm{d}\tau\right]^2\Bigg).\nonumber
\end{eqnarray}
From the Brownian motion properties, we known that $\mathbb{E}[\xi_{ij}(t)\xi_{kl}(t)] = 0$ for $k \ne i, l \ne j$ and $\mathbb{E}\xi_{ij}^2(t) = t$. By the H\"{o}lder inequality, we obtain 
\begin{eqnarray}
\mathbb{E}(I_{1,1}) &\le& 2\sum_{p=1}^N \sum_{q=1}^M\left(\dfrac{nm}{\pi^2}V_{\text{max}}\lambda^{-2}_{p,q}(A(T)) + \sum_{i=1}^n\sum_{j=1}^m\phi^2_{p,q}(x_i,y_j)\int_0^T\lambda^{-2}_{p,q}(A(\tau))\mathrm{d}\tau\int_{0}^{T}\vartheta^2\mathbb{E}\xi^2_{ij}(\tau) \mathrm{d}\tau\right)\nonumber\\
&\le& 2\sum_{p=1}^N \sum_{q=1}^M\left(\dfrac{nm}{\pi^2}V_{\text{max}}\lambda^{-2}_{p,q}(A(T)) + \dfrac{\vartheta^2T^2}{2}\int_0^T\lambda^{-2}_{p,q}(A(\tau))\mathrm{d}\tau\sum_{i=1}^n\sum_{j=1}^m\phi^2_{p,q}(x_i,y_j)\right)\nonumber\\
&\le& 2\sum_{p=1}^N \sum_{q=1}^M\left(\dfrac{nm}{\pi^2}V_{\text{max}}\lambda^{-2}_{p,q}(A(T)) + \dfrac{\vartheta^2T^2 nm}{2\pi^2}\int_0^T\lambda^{-2}_{p,q}(A(T))\mathrm{d}\tau \right) \nonumber\\
&\le& \dfrac{2nm}{\pi^2}\left(V_{\text{max}} + \dfrac{\vartheta^2T^3}{2}\right)\sum_{p=1}^N \sum_{q=1}^Me^{2A(T)(p^2+q^2)}.\nonumber
\end{eqnarray}
According Lemma \ref{appro_int1}, we have
\begin{eqnarray}
\mathbb{E}(I_{1,1}) &\le& \dfrac{2nm}{\pi^2}\left(V_{\text{max}} + \dfrac{\vartheta^2T^3}{2}\right)\int_{1}^{N+1}\int_{1}^{M+1}e^{2A(T)(s^2+r^2)}\mathrm{d}r\mathrm{d}s \nonumber\\
&\le& \dfrac{2nm}{\pi^2}\left(V_{\text{max}} + \dfrac{\vartheta^2T^3}{2}\right)\int_{1}^{N+1}e^{2A(T)s^2}\mathrm{d}s\int_{1}^{M+1}e^{2A(T)r^2}\mathrm{d}r.\nonumber
\end{eqnarray}
Noting that $e^{N^2}\leq n^{\frac{\omega_1}{4A(T)}}, e^{M^2}\leq m^{\frac{\omega_2}{4A(T)}}$, we obtain
\begin{eqnarray}
\mathbb{E}(I_{1,1}) &\le& \dfrac{nm\left(V_{\text{max}} + \dfrac{\vartheta^2T^3}{2}\right)}{2A^2(T)(N+1)(M+1)}e^{2A(T)[(N+1)^2+(M+1)^2]}\nonumber\\
&\le& \dfrac{n^{\frac{\omega_1}{2}+1}m^{\frac{\omega_2}{2}+1}\left(V_{\text{max}} + \dfrac{\vartheta^2T^3}{2}\right)}{2A^2(T)(N+1)(M+1)}e^{2A(T)[2N+2M]}.\nonumber
\end{eqnarray}
Putting $\eta_{n,m}=\max\{|\eta_{n,m,p,q}|:\ p=\overline{1,N}, q=\overline{1,M}\}$, we obtain directly
\begin{eqnarray}
I_{1,2} &\le& \sum_{p=1}^N \sum_{q=1}^M\eta_{n,m}^2\left(\int_0^T\lambda^{-1}_{p,q}(\tau)\mathrm{d}\tau\right)^2 \nonumber\\
&\le& \eta_{n,m}^2\sum_{p=1}^N \sum_{q=1}^M\left[\int_0^T e^{(p^2+q^2)\int_0^\tau a(s)\mathrm{d}s} \mathrm{d}\tau\right]^2 \nonumber\\
&\le& \eta_{n,m}^2\sum_{p=1}^N \sum_{q=1}^M\left(\int_0^T e^{a_2\tau(p^2+q^2)} \mathrm{d}\tau\right)^2. \nonumber
\end{eqnarray}
Hence, it follows from Lemma \ref{eta} that
\begin{eqnarray}
I_{1,2}&\le& \eta_{n,m}^2\sum_{p=1}^N \sum_{q=1}^M\dfrac{e^{2a_2T[p^2+q^2]}}{a_2^2(p^2+q^2)^2} \le \dfrac{e^{2a_2T[(N+1)^2 + (M+1)^2]}}{2a_2^3T(N+1)(M+1)}\eta_{n,m}^2\nonumber\\
&\le& C'^2\left(n^{-\alpha}+m^{-\beta}\right)^2 \dfrac{2n^{\frac{\omega_1}{2}}m^{\frac{\omega_2}{2}}}{a_2^3T(N+1)(M+1)}e^{2A(T)[2N+2M]}.\nonumber
\end{eqnarray}
Now, we find an upper bound of $I_{1,3}$. Putting $\gamma_{n,m}=\max\{|\gamma_{n,m,p,q}|:\ p\in\overline{1,N},
q\in\overline{1,M}\}$  and using Lemma \ref{gamma} we have
\begin{eqnarray}
I_{1,3} &\le& \gamma_{n,m}^2 \lambda_{p,q}^{-2}(T) \le 8\sum_{p=1}^N \sum_{q=1}^M(K_{1,n,m}^2 + K_{2,n,m}^2)\lambda_{p,q}^{-2}(T), \nonumber
\end{eqnarray}
where $K_{1,n,m}, K_{2,n,m}$ are defined in \eqref{gamma1_1}, \eqref{gamma1_2}.\\
We get
\begin{eqnarray}
\sum_{p=1}^N \sum_{q=1}^M K_{1,n,m}^2\lambda_{p,q}^{-2}(T) &\le& \sum_{p=1}^N \sum_{q=1}^M \left[4e^{-a_1T(2n-p+q^2)}\|\theta\| + \dfrac{2E}{a_1n^{2+\alpha}}\right]^2\lambda_{p,q}^{-2}(T)\nonumber\\
&\le& \dfrac{8E^2}{a_1^2n^{4+2\alpha}} \sum_{p=1}^N \sum_{q=1}^M\dfrac{\lambda_{p,q}^{-2}(T)}{p^{2+2\alpha}}+ 32e^{-4na_1T}\|\theta\|^2\sum_{p=1}^N \sum_{q=1}^Me^{-2a_1T(q^2 - p)} \lambda_{p,q}^{-2}(T)\nonumber\\
&\le& \dfrac{8E^2}{a_1^2n^{4+2\alpha}}\sum_{p=1}^N \sum_{q=1}^M e^{2A(T)[p^2+q^2]}\nonumber\\ && + \: 32e^{-4na_1T}\|\theta\|^2\sum_{p=1}^N \sum_{q=1}^Me^{2(A(T)p^2 + a_1Tp)}e^{2q^2(A(T)-a_1T)}\nonumber\\
&\le& \dfrac{8E^2 e^{2A(T)[(N+1)^2 + (M+1)^2]}}{a_1^2n^{4+2\alpha} A(T)(N+1)(M+1)}\nonumber\\ && + \: 32e^{-4na_1T}NMe^{2A(T)(N^2+M^2)}e^{2a_1T(N - M^2)}\|\theta\|^2\nonumber\\
&\le& \dfrac{4E^2 n^{\frac{\omega_1}{2} - 4 - \alpha}m^{\frac{\omega_2}{2}}}{a_1^2A(T)(N+1)(M+1)} + 64e^{-4na_1T}n^{\frac{\omega_1}{2}}m^{\frac{\omega_2}{2}}NM\|\theta\|^2. \nonumber
\end{eqnarray}
Similarly, we obtain
\begin{eqnarray}
\sum_{p=1}^N \sum_{q=1}^M K_{2,n,m}^2\lambda_{p,q}^{-2}(T) 
&\le& \dfrac{4E^2 m^{\frac{\omega_2}{2} - 4 - \beta}n^{\frac{\omega_1}{2}}}{a_1^2A(T)(N+1)(M+1)} + 64e^{-4ma_1T}n^{\frac{\omega_1}{2}}m^{\frac{\omega_2}{2}}NM\|\theta\|^2.\nonumber
\end{eqnarray}
Hence,
\begin{eqnarray}
I_{1,3} &\le& 4n^{\frac{\omega_1}{2}}m^{\frac{\omega_2}{2}} \Bigg[\dfrac{E^2\left(n^{- 4 - \alpha} + m^{- 4 - \beta}\right)}{a_1^2A(T)(N+1)(M+1)} + 16MN\left(e^{-4ma_1T} + e^{-4na_1T}\right)\|\theta\|^2\Bigg].\nonumber
\end{eqnarray}
Therefore, we get
\begin{eqnarray}
\mathbb{E}I_1 &\le& e^{2A(T)[2N+2M]}\left[\dfrac{6\pi^2 n^{\frac{\omega_1}{2}-1}m^{\frac{\omega_2}{2}-1}\left(V_{\text{max}} + {\vartheta^2 T^3}/{2}\right)}{2A^2(T)(N+1)(M+1)} +\right.\\
& &\left. C'^2\left(n^{-\alpha}+m^{-\beta}\right)^2\dfrac{24n^{\frac{\omega_1}{2}-2}m^{\frac{\omega_2}{2}-2}}{a_2^3T(N+1)(M+1)}\right]+\nonumber\\ &&  \:  48\pi^4n^{\frac{\omega_1}{2}-2}m^{\frac{\omega_2}{2}-2}\Bigg[\dfrac{E^2\left(n^{- 4 - \alpha} + m^{- 4 - \beta}\right)}{a_1^2A(T)(N+1)(M+1)} + 16MN\left(e^{-4ma_1T} + e^{-4na_1T}\right)\|\theta\|^2\Bigg]\nonumber\\
&=& C\Delta_{n,m,\omega_1,\omega_2},
\label{E_I1}
\end{eqnarray}
where 
\begin{equation*}
 \Delta_{n,m,\omega_1,\omega_2}=4E^2\left[\left(\dfrac{\omega_1
\log n}{4A(T)}\right)^{-\alpha} + \left(\dfrac{\omega_2
\log m}{4A(T)}\right)^{-\beta} \right].
\label{E_I2}
\end{equation*}
To finish the proof of this theorem, we find an upper bound for $I_2$. In fact, we have
\begin{eqnarray}
I_2 &\le& 4\left(\sum_{p=N+1}^\infty \sum_{q=1}^M \left|A^2_{p,q}\right| + \sum_{p=1}^N \sum_{q=M+1}^\infty \left|\theta_{p,q}^2\right| + \sum_{p=N+1}^\infty \sum_{q=M+1}^\infty \left|\theta_{p,q}^2\right|\right)\nonumber\\
&\le& 4\Bigg(\sum_{p=N+1}^\infty \sum_{q=1}^M p^{-2\alpha}q^{-2\beta}\left|\langle p^\alpha q^\beta \theta, \phi_{p,q}\rangle\right|^2 + \sum_{p=1}^N \sum_{q=M+1}^\infty p^{-2\alpha}q^{-2\beta}\left|\langle p^\alpha q^\beta \theta, \phi_{p,q}\rangle\right|^2 \nonumber\\ && + \: \sum_{p=N+1}^\infty \sum_{q=M+1}^\infty p^{-2\alpha}q^{-2\beta}\left|\langle p^\alpha q^\beta \theta, \phi_{p,q}\rangle\right|^2\Bigg)\nonumber\\
&\le& 4E^2\left(N^{-2\alpha}+M^{-2\beta} + N^{-2\alpha}M^{-2\beta}\right)\nonumber\\
&\le & 2\Delta_{n,m,\omega_1,\omega_2}.
\end{eqnarray}
Therefore there exists a positive number $C_0$ independent of $n,m,N,M$ such that
\[\mathbb{E}\left\|\hat{\theta}_{n,m,N,M} - \theta \right\|^2 \le C_0\left[\left(\dfrac{\omega_1
\log n}{4A(T)}\right)^{-\alpha} + \left(\dfrac{\omega_2
\log m}{4A(T)}\right)^{-\beta}\right].\]
\end{proof}
\section{Numerical Results}
We illustrate the theoretical results by concrete examples. To this end, we first describe a 
plan for computation. Let $\Omega = (0,\pi) \times (0,\pi)$, $T=1$ and 
\begin{equation}\label{exam}
\left\{\begin{array}{l l l}
  	u_t - a(t)\Delta u  & = f(x,y,t), & \Omega\times (0,1),\\
  	u(x,y,t)\big|_{\partial \Omega} & = 0, & 0 \le t \le 1,\\
  	u(x,y,1) & = h(x,y), & (x,y) \in\overline{\Omega},
\end{array}\right.\nonumber
\end{equation}
where the functions $f(x,y,t),h(x,y)$ are measured and the function $a:[0,1]\to\mathbb{R}$  is known.
 
We shall simulate the data for heat source term and final condition, respectively. In fact, at each point $(x_i,y_j)=\left(\frac{\pi (2i-1)}{2n},\frac{\pi (2j-1)}{2m}\right)$, $i=\overline{1,n}, j=\overline{1,m}$, using two subroutines in FORTRAN programs of John Barhardt (see \cite{r22}) and of Marsaglia G., Tsang W. W. (see \cite{r23}), we make noises  the heat source by   $\vartheta\xi_{ij}(t)$ and  the final data  by  $\sigma_{ij}\epsilon_{ij}$ where  $\xi_{ij}(t)$ are the
normal Brownian motions and $\epsilon_{ij}$ are the standard normal random variables. Choosing $\sigma_{ij}^2 = \sigma^2 = \vartheta = 10^{-1}$ and $10^{-2}$, we have two following regression models
\begin{eqnarray}
d_{ij} &=& h(x(i),y(j)) + \sigma\epsilon_{ij}, \qquad \epsilon_{ij} \stackrel{\text{i.i.d}}{\sim}\mathcal{N}(0,1),\nonumber\\
g_{ij}(t) &=& f(x(i),y(j)) + \vartheta\xi_{ij}.\nonumber
\end{eqnarray}

Now, we choose some numerical methods to compare errors. The first method is the trigonometric nonparametric regression (truncated method for short) which is considered in the present paper. The second method is the quasi--boundary value (QBV) regularization. The third method is based on the classical solution (CS for short) of the backward problem. 

For the mentioned function $a$, we use the method Legendre--Gauss quadrature with the roots $x_i$ of the Legendre polynomials $P_{512}(x),x\in[-1,1]$ to calculate
\[A_{GL}=\int_{0}^{1}a(s)\mathrm{d}s=\frac{1}{2}\sum_{n=1}^{512}w_ia\left(\frac{x_i}{2}+\frac{1}{2}\right)\]
where 
$$w_i=\frac{2}{(1+x_i^2)\left[P'_{512}(x_i)\right]^2}.$$

The first method is the truncated one which is considered in the present paper. In the method, we have to set up the values of $N,M$.
With the quantity $A_{GL}$, we can obtain the values of $N,M$ from $n,m$ and $\omega_1 = \omega_2 = 1$ by the following formula
\[N = \left\lfloor\dfrac{(\log n)^{1/2}}{A_{GL}}\right\rfloor \text{ and } M = \left\lfloor\dfrac{(\log m)^{1/2}}{A_{GL}}\right\rfloor.
\]
In each case of variance $\sigma_{ij}^2 = \sigma^2$, we compute $30$ times. To calculate the error between the exact solution and the estimator, we use the root mean squared error (RMSE) as follows
\[\mathrm{RMSE}(\hat{\theta};\theta) = \sqrt{\frac{1}{nm}\sum_{i=1}^{n}\sum_{j=1}^{m}\left(\hat{\theta}(x_i,y_j) - \theta(x_i,y_j)\right)^2}.\]
Then, we find the average of $\mathrm{RMSE}(\hat{\theta};\theta)$ in 30 runs order.\\ 

The second method is the quasi--boundary value (QBV) regularization with the approximation of the initial data
\[\theta_{QBV}(x,y)=\sum_{p=1}^{\infty}\sum_{q=1}^{\infty}\left(\frac{\hat{h}_{p,q}}{\epsilon(p^2+q^2)+\lambda_{p,q}(T)}-\int_{0}^{T}\frac{\lambda^{-1}_{p,q}(\tau)\lambda_{p,q}(T)}{\epsilon(p^2+q^2)+\lambda_{p,q}(T)}\hat{f}_{p,q}(\tau)\text{d}\tau\right)\phi_{p,q}(x,y).\]
The method is chosen since it is quite common and the stability magnitude of the regularization operator is of order $O(\epsilon^{-1})$ (see \cite{show}). As mentioned, in the QBV method, we do not have explicit stopping indices. So, we only calculate 
 with $p,q=\overline{1,20}; \epsilon=\sigma^2$ and use the formula
\[\theta_{QBV}(x,y)\approx\sum_{p=1}^{20}\sum_{q=1}^{20}\left(\frac{\hat{h}_{p,q}}{\epsilon(p^2+q^2)+\lambda_{p,q}(T)}-\int_{0}^{T}\frac{\lambda^{-1}_{p,q}(\tau)\lambda_{p,q}(T)}{\epsilon(p^2+q^2)+\lambda_{p,q}(T)}\hat{f}_{p,q}(\tau)\text{d}\tau\right)\phi_{p,q}(x,y).\]

Finally, we consider a numerical result for the classical solution (CS for short). As the second method, we use the approximation formula
\[\theta_{CS}(x,y)\approx\sum_{p=1}^{20}\sum_{q=1}^{20}\left(\hat{h}_{p,q}\lambda^{-1}_{p,q}(T)-\int_{0}^{T}\lambda^{-1}(\tau)\hat{f}_{p,q}(\tau)\text{d}\tau\right)\phi_{p,q}(x,y).\] 
We shall illustrate the discussed plan by two examples. In Example 1, we consider the problem with
an exact initial datum $\theta$ having a finite Fourier expansion. In Example 2, we compute with the function $\theta$ having an infinite Fourier expansion.

In the examples, to calculate integrals depended on the time variable $t$ in approximation formulae, we use the generalized Simpson approximation with 101 equidistant points $0=t_0<t_1<\dots<t_{101}=1$
\[\int_0^1 \nu(\tau) \mathrm{d}\tau=\frac{1}{100}\left[\frac{3}{8}\nu(t_0)+\frac{7}{6}\nu(t_1)+\frac{23}{24}\nu(t_2)+\sum_{k=3}^{n-3}\nu(t_k)+\frac{23}{24}\nu(t_{99}+\frac{7}{6}\nu(t_{100})+\frac{3}{8}\nu(t_{101})\right]\] where $\nu(\tau)=\lambda_{p,q}^{-1}(\tau)\hat{f}_{p,q}(\tau)$.\\
 
\textit{Example 1.} With $a(t)=2-t$, we can see that $1=a_1\le a(t)\le a_2=2$. We have $A_{GL}=1.5$.
Assuming $f(x,y,t) = 2(t^3 - 2t^2 - 6t + 10)\sin(x)\sin(y)$ and $h(x,y) = 4\sin(x)\sin(y)$. The exact value of $u(x,y,0)$ is 
\[\theta(x,y) = 5\sin(x)\sin(y)\]
which has a finite Fourier expansion.

Figure \ref{fg:data} and Figure \ref{fg:data_f} present surfaces of the data and their contours without and within noises for the final condition and the source term. They are drawn in case $\sigma^2 = 10^{-1}$, $n = m = 81$ and at the time $t = 0.5$, w.r.t. \\
\begin{figure}[btbp]
\begin{center}
\begin{tabular}{@{}c@{}c@{}}
  \multicolumn{2}{c}{\footnotesize The data set of the final temperature}\\[-0.2em]
   \includegraphics[width=8cm, height=6cm]{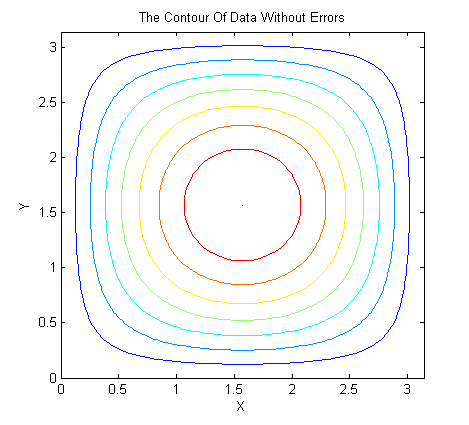}&
   \includegraphics[width=8cm, height=6cm]{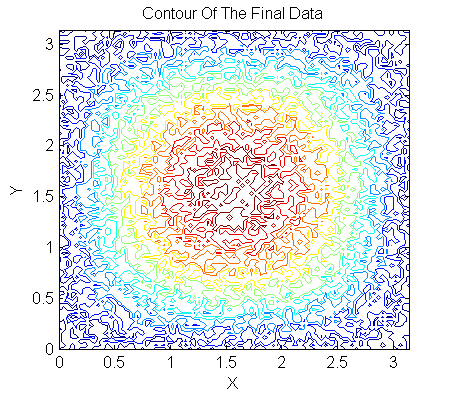}\\[-0.1em]
  \footnotesize Without Noises & \footnotesize Within Noises\\
\end{tabular}
\caption{The Contour of Two Data Set for Final Temperature.}
\label{fg:data}
\end{center}
\end{figure}
According to the figures, we can see the non--smoothness of two surfaces data in case of random noise. In fact, from the contour plot within noise of the final data, we also see that the measured data is very chaotic.
\begin{figure}[htbp]
\begin{center}
\begin{tabular}{@{}c@{}c@{}}
  \multicolumn{2}{c}{\footnotesize The data set of the source term at $t=0.5$}\\[-0.2em]
  \includegraphics[width=8cm, height=6cm]{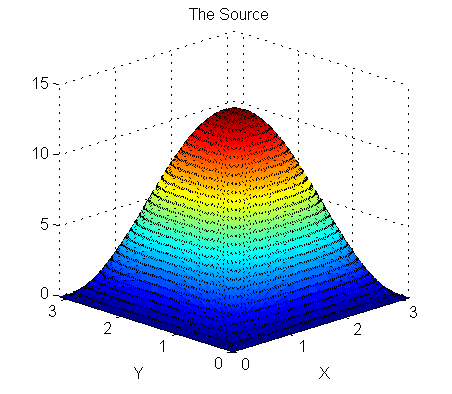}&
  \includegraphics[width=8cm, height=6cm]{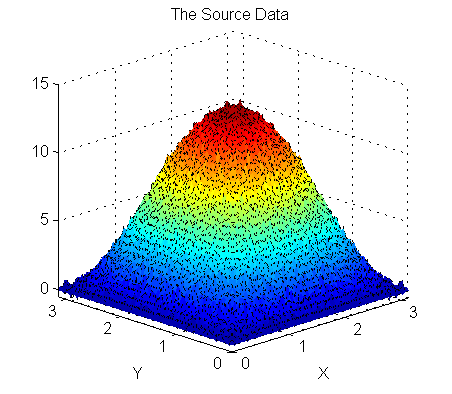}\\
  \multicolumn{2}{c}{\footnotesize }\\[-0.2em]    
  \footnotesize Without Noises & \footnotesize Within Noises\\
\end{tabular}
\caption{The Surface of Data Set for Heat Source.}
\label{fg:data_f}
\end{center}
\end{figure}
In case of $\sigma^2 = 10^{-1}$, the error of the estimation is quite large, while, the error in case of $\sigma^2 = 10^{-2}$ is smaller. In addition, we see that the errors (in two cases of the variance $\sigma_{ij}^2 = \sigma^2$) are decreased when $n,m$ are increased (see Figures \ref{fg:er1}). The results of this experiment have demonstrated numerically the effectively of the estimator.
\begin{table}[htbp]
	\caption{Comparing errors between methods in Example 1: $\sigma^2 = 10^{-1}, 10^{-2}$ and $n=m=21$.}
	\begin{center}
		\begin{tabular}{l c c c c c c }
			\toprule 
			Run & \multicolumn{2}{c}{Estimator} & \multicolumn{2}{c}{QBV method} & \multicolumn{2}{c}{Classical solution} \\
			& $\sigma^2=10^{-1}$ & $\sigma^2=10^{-2}$ 
			& $\epsilon=10^{-1}$ & $\epsilon=10^{-2}$ & $\epsilon=10^{-1}$ & $\epsilon=10^{-2}$ \\
			\midrule
				1 & 0.3488 & 0.0855 & 1.8493 & 0.6836 & 9.0696E+0466 & 7.2832E+0467 \\
				2 & 0.2810 & 0.0098 & 1.7936 & 0.6492 & 5.6003E+0468 & 1.3249E+0467 \\
				3 & 0.1665 & 0.1151 & 1.6715 & 0.6741 & 4.9925E+0468 & 7.6606E+0467 \\
				4 & 0.0642 & 0.0555 & 1.8313 & 0.6199 & 1.9484E+0468 & 8.8691E+0466 \\
				5 & 0.3478 & 0.0795 & 1.7854 & 0.5895 & 3.0650E+0468 & 9.9884E+0467 \\
				6 & 0.1541 & 0.1344 & 1.7437 & 0.6661 & 1.6817E+0468 & 1.0375E+0466 \\
				7 & 1.1359 & 0.1045 & 1.9001 & 0.6162 & 1.0333E+0468 & 5.0317E+0467 \\
				8 & 0.1819 & 0.1116 & 1.8155 & 0.6789 & 4.4777E+0468 & 2.6705E+0467 \\
				9 & 0.5098 & 0.0794 & 1.9957 & 0.6704 & 8.7766E+0467 & 1.9412E+0467 \\
				10 & 0.0767 & 0.0819 & 1.7344 & 0.6770 & 1.9678E+0468 & 7.3191E+0466 \\
				11 & 0.6926 & 0.0509 & 1.8346 & 0.6305 & 2.8522E+0468 & 3.6677E+0467 \\
				12 & 0.1562 & 0.0650 & 1.8199 & 0.6876 & 9.8178E+0468 & 6.0419E+0467 \\
				13 & 0.3010 & 0.0133 & 1.6247 & 0.6591 & 1.3412E+0468 & 4.7005E+0467 \\
				14 & 0.2691 & 0.0549 & 1.9827 & 0.6664 & 4.9153E+0468 & 2.6146E+0466 \\
				15 & 0.8242 & 0.0784 & 1.8294 & 0.6782 & 2.8401E+0468 & 2.2989E+0467 \\
				16 & 0.0800 & 0.0897 & 2.0291 & 0.6365 & 3.9761E+0468 & 3.2519E+0467 \\
				17 & 0.5340& 0.0694 & 1.8317 & 0.6593 & 5.5066E+0466 & 4.5486E+0467 \\
				18 & 0.3112 & 0.0560 & 1.7623 & 0.6140 & 5.6634E+0468 & 4.9512E+0467 \\
				19 & 0.0823 & 0.1052 & 1.8327 & 0.6706 & 4.7594E+0467 & 1.5004E+0467 \\
				20 & 0.8982 & 0.0593 & 1.7463 & 0.6531 & 4.2411E+0468 & 3.3806E+0467 \\
				21 & 1.1967 & 0.0919 & 1.8337 & 0.6322 & 6.7184E+0468 & 3.4589E+0467 \\
				22 & 0.6456 & 0.1117 & 1.6554 & 0.6898 & 3.1764E+0468 & 8.5158E+0467 \\
				23 & 0.7978 & 0.0921 & 1.8755 & 0.6289 & 1.9857E+0468 & 1.3291E+0467 \\
				24 & 0.7382 & 0.0732 & 1.8330 & 0.6568 & 1.4733E+0468 & 1.6599E+0467 \\
				25 & 0.2039 & 0.1161 & 1.7400 & 0.6372 & 2.2766E+0468 & 2.3429E+0467 \\
				26 & 0.1441 & 0.1000 & 1.8158 & 0.6410 & 9.6333E+0467 & 3.4518E+0467 \\
				27 & 1.3111 & 0.1097 & 1.7632 & 0.6621 & 2.3796E+0468 & 3.9224E+0467 \\
				28 & 0.3626 & 0.1020 & 1.8254 & 0.6583 & 4.7331E+0468 & 7.5024E+0466 \\
				29 & 0.2833 & 0.0173 & 1.7640 & 0.6552 & 8.1452E+0467 & 1.4300E+0467 \\
				30 & 0.8313 & 0.0414 & 1.9595 & 0.6774 & 4.0568E+0468 & 4.4984E+0467	\\
			\midrule   
			Average &	0.4643	&	0.0785	&	1.8160	&	0.6540	&	divergence	&	divergence	\\
			\bottomrule
		\end{tabular}
		\label{table:er1}
	\end{center}
\end{table}
Table \ref{table:er1} shows the error of the method. We see that the error between the exact solution with the classical solution grows very fast. In fact, the error data is quite small $\epsilon=10^{-1},10^{-2}$ but the error solution is large $\approx10^{466}$. This illustrates numerically the ill--posedness of  our problem. The other hand, the errors in Table \ref{table:er1} of the truncated method is better than the one of the QBV method.
\begin{figure}[htbp]
	\begin{center}
		\begin{tabular}{@{}c@{}c@{}}
			\multicolumn{2}{c}{\footnotesize }\\[-0.2em]
			\includegraphics[width=8cm, height=6cm]{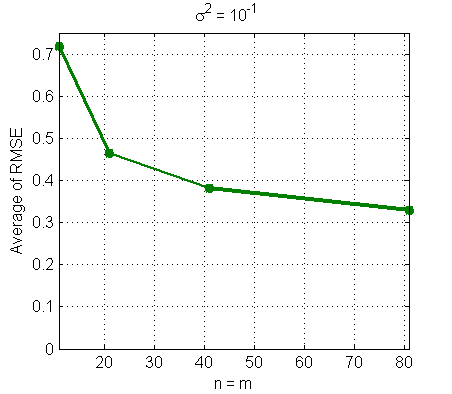}&
			\includegraphics[width=8cm, height=6cm]{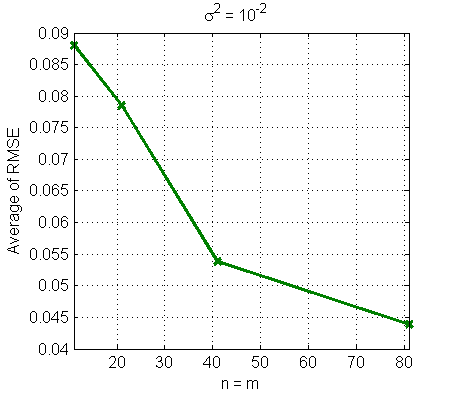}\\[-0.1em]
		\end{tabular}
		\caption{The Graphics of The Average of RMSE in two cases $\sigma^2 = 10^{-1}$ and \small $\sigma^2 = 10^{-2}$.}
		\label{fg:er1}
	\end{center}
\end{figure}
\textit{Example 2.}
Let $a(t)=0.5e^{-t}$ and $e^{-1}=a_1\le a(t)\le a_2=1$. Then, we calculate $A_{GL}=0.3161$. Suppose that
\begin{eqnarray*}
	f(x,y,t) &=& \frac{e^{-t}}{\pi}\left[\left(2e^{-t}+(4e^{-t}-1)\sin 2y\right)+(1-10e^{-t})\sin 3x\sin y\right],\\
	h(x,y) &=& 	\frac{e^{-1}}{\pi}\left[x(\pi-x)\sin y-\sin 3x\sin y\right].
\end{eqnarray*} 
We easily see that the exact value of $u(x,y,0)$ is
\[\theta(x,y) = \frac{1}{\pi}\left[x(\pi-x)\sin y-\sin 3x\sin y\right]\]
which has an infinite Fourier expansion.

The results of Example 2 error as in Table \ref{table:er2}. From the results, we can obtain the same conclusions as in Example 1.
%
%
%
\begin{table}[htbp]
	\caption{Comparing errors between methods in Example 2: $\sigma^2 = 10^{-1}, 10^{-2}$ and $n=m=21$.}
	\begin{center}
		\begin{tabular}{l c c c c c c }
			\toprule 
			Run & \multicolumn{2}{c}{Estimator} & \multicolumn{2}{c}{QBV method} & \multicolumn{2}{c}{Classical solution} \\
			& $\sigma^2=10^{-1}$ & $\sigma^2=10^{-2}$ 
			& $\epsilon=10^{-1}$ & $\epsilon=10^{-2}$ & $\epsilon=10^{-1}$ & $\epsilon=10^{-2}$ \\
			\midrule
1  & 0.2702                & 0.1533                & 0.2518                & 0.1431                & 4.79E+097             & 4.55E+095             \\
2  & 0.2111                & 0.1536                & 0.3038                & 0.1512                & 4.97E+096             & 1.22E+096             \\
3  & 0.1865                & 0.1540                & 0.2881                & 0.1402                & 9.17E+094             & 2.98E+096             \\
4  & 0.3827                & 0.1539                & 0.3039                & 0.1350                & 3.80E+097             & 2.34E+096             \\
5  & 0.2872                & 0.1525                & 0.3161                & 0.1521                & 1.18E+096             & 2.27E+096             \\
6  & 0.2492                & 0.1564                & 0.3135                & 0.1525                & 7.43E+096             & 1.23E+096             \\
7  & 0.2468                & 0.1539                & 0.2858                & 0.1289                & 1.28E+097             & 1.30E+096             \\
8  & 0.6985                & 0.1531                & 0.2876                & 0.1436                & 2.11E+097             & 3.79E+095             \\
9  & 0.2923                & 0.1534                & 0.3187                & 0.1421                & 2.04E+097             & 2.85E+095             \\
10 & 0.3177                & 0.1549                & 0.3104                & 0.1484                & 4.02E+097             & 1.14E+096             \\
11 & 0.1931                & 0.1534                & 0.2909                & 0.1386                & 1.21E+097             & 6.71E+095             \\
12 & 0.1957                & 0.1563                & 0.3355                & 0.1821                & 1.60E+097             & 6.91E+095             \\
13 & 0.1964                & 0.1532                & 0.3139                & 0.1313                & 7.14E+096             & 1.68E+096             \\
14 & 0.2875                & 0.1553                & 0.3018                & 0.1570                & 3.96E+096             & 1.24E+096             \\
15 & 0.2700                & 0.1540                & 0.3195                & 0.1403                & 4.51E+097             & 2.42E+096             \\
16 & 0.2558                & 0.1545                & 0.2985                & 0.1362                & 1.68E+097             & 5.06E+096             \\
17 & 0.1976                & 0.1535                & 0.3589                & 0.1466                & 1.60E+096             & 2.91E+096             \\
18 & 0.4981                & 0.1548                & 0.3853                & 0.1594                & 3.85E+096             & 2.65E+096             \\
19 & 0.2723                & 0.1539                & 0.3200                & 0.1540                & 9.70E+096             & 2.71E+096             \\
20 & 0.3152                & 0.1534                & 0.3312                & 0.1466                & 2.04E+097             & 1.77E+096             \\
21 & 0.3284                & 0.1544                & 0.3303                & 0.1442                & 1.60E+097             & 1.76E+096             \\
22 & 0.4009                & 0.1526                & 0.3173                & 0.1274                & 1.01E+097             & 3.65E+096             \\
23 & 0.3175                & 0.1532                & 0.3005                & 0.1475                & 9.57E+096             & 1.65E+096             \\
24 & 0.4426                & 0.1526                & 0.3132                & 0.1537                & 3.06E+096             & 8.29E+095             \\
25 & 0.3158                & 0.1528                & 0.3234                & 0.1353                & 4.21E+097             & 2.98E+096             \\
26 & 0.2715                & 0.1545                & 0.3443                & 0.1428                & 1.22E+097             & 2.06E+096             \\
27 & 0.1848                & 0.1527                & 0.3517                & 0.1312                & 2.46E+097             & 2.31E+096             \\
28 & 0.2695                & 0.1555                & 0.3104                & 0.1470                & 1.05E+097             & 1.81E+096             \\
29 & 0.5497                & 0.1637                & 0.3126                & 0.1300                & 7.88E+096             & 4.31E+096             \\
30 & 0.3161                & 0.1530                & 0.3047                & 0.1414                & 4.69E+096             & 8.49E+095             \\			\midrule   
			Average &	0.3074	&	0.1542 	&	0.3148	&	0.1443 	&	divergence	&	divergence	\\
			\bottomrule
		\end{tabular}
		\label{table:er2}
	\end{center}
\end{table}
\section{Conclusion}
In this paper, we consider a nonhomogeneous backward problem with initial data and source having
random noises. We have to estimate the initial data and the source by regression methods in statistics. On the other hand, our problem is ill-posed. Hence, a regularization is in order.
We have used  the trigonometric method in nonparametric regression  associated with the truncated expansion method to approximate stably the Fourier coefficients of the unknown function $\theta(x,y)$. The estimate of  bias of the discretization is given explicitly. Finally, we illustrate the theoretical part by comparing computation results of  nonparametric regression, QBV and classical solution methods.   

\section*{Acknowledgments}
We would like to express our sincere thanks to the anonymous referees for constructive comments that 
improved a lot of  idea in our paper.

\end{document}